\documentclass[letterpaper, 11pt]{amsart}

\usepackage{amsthm}
\usepackage{amssymb}
\usepackage{amsmath}
\usepackage{commath}
\usepackage[margin=1in]{geometry}

\newtheorem{prop}{Proposition}
\newtheorem{lemma}{Lemma}
\newtheorem{thm}{Theorem}
\newtheorem{cor}{Corollary}

\DeclareMathOperator{\ord}{ord}
\DeclareMathOperator{\Res}{Res}

\DeclareMathOperator{\PGL}{PGL}

\DeclareMathOperator{\BDV}{BDV}
\DeclareMathOperator{\CPA}{CPA}

\DeclareMathOperator{\supp}{supp}
\DeclareMathOperator{\Wr}{Wr}
\DeclareMathOperator{\Lip}{Lip}
\DeclareMathOperator{\diam}{diam}

\DeclareMathOperator{\hberk}{\mathrm{\bold H}^1}
\DeclareMathOperator{\aberk}{\mathrm{\bold A}^1}
\newcommand{\zetaG}{\zeta_{\text{G}}}

\DeclareMathOperator{\pberk}{\mathrm{\bold P}^1}

\newcommand{\Lips}{\mathcal{L}}
\DeclareMathOperator{\diamG}{\textrm{diam}_{\zetaG}}

\newcommand{\vv}{\vec{v}}
\newcommand{\vw}{\vec{w}}

\newcommand{\vz}{\vec{z}}

\newcommand{\PP}{\mathbb{P}}
\newcommand{\RR}{\mathbb{R}}

\newcommand{\CC}{\mathbb{C}}

\title{Lower Bounds for non-Archimedean Lyapunov Exponents}

\author{Kenneth Jacobs}

\begin{document}
\maketitle
\begin{abstract}
Let $K$ be a complete, algebraically closed, non-Archimedean valued field, and let $\pberk$ denote the Berkovich projective line over $K$. The Lyapunov exponent for a rational map $\phi\in K(z)$ of degree $d\geq 2$ measures the exponential rate of growth along a typical orbit of $\phi$. When $\phi$ is defined over $\CC$, the Lyapunov exponent is bounded below by $\frac{1}{2}\log d$. In this article, we give a lower bound for $L(\phi)$ for maps $\phi$ defined over non-Archimedean fields $K$. The bound depends only on the degree $d$ and the Lipschitz constant of $\phi$. For maps $\phi$ whose Julia sets satisfy a certain boundedness condition, we are able to remove the dependence on the Lipschitz constant.
\end{abstract}

\section{Introduction}
Let $K$ be an algebraically closed field that is complete with respect to a non-Archimedean absolute value $|\cdot|_v$. Let $\mathcal{O} = \{z\in K : |z|_v \leq 1\}$ be its ring of integers, $\mathfrak{m} = \{z\in K : |z|_v <1\}$ its maximal ideal, and $k = \mathcal{O}/\mathfrak{m}$ the residue field. Let $q_v = e$ if $\textrm{char}(k) = 0$; otherwise let $q_v$ be the characteristic of $k$. We will adopt the notation $\log_v x= \log_{q_v} x$.

The aim of this article is to give a lower bound for the Lyapunov exponent $L(\phi)$ of a rational map $\phi\in K(T)$ of degree $d\geq 2$. For a map $\phi\in \CC(z)$ defined over the complex numbers, the Lyapunov exponent can be given $$L(\phi)= \int_{\PP^1(\CC)} \log \phi^\#\ d\mu_\phi\ ,$$ where $\phi^\#$ is the derivative of $\phi$ in the chordal metric on $\PP^1(\CC)$ and $\mu_\phi$ is the measure of maximal entropy for $\phi$. A sharp lower bound for the Lyapunov exponent over $\CC$ is $ L(\phi)\geq \frac{1}{2} \log d (>0)$ (see \cite{Ly}, \cite{MSS}, \cite{Ru}). It is a striking feature of non-Archimedean dynamics, then, that when $\phi$ is defined over a non-Archimedean field, the Lyapunov exponent can be negative. This is the case, for example, when $\phi(T) = T^p$ for some prime number $p$ and $K=\CC_p$.\\

Let $\phi^\#$ denote the derivative of $\phi$ on $\mathbb{P}^1(K)$ with respect to the chordal metric. The space $\PP^1(K)$ is totally disconnected in its analytic topology, and so we embed it into the Berkovich projective line $\pberk$ over $K$. The Berkovich projective line is a compact, Hausdorff space which is uniquely path connected and which contains $\PP^1(K)$ as a dense subset. Formally, elements of $\pberk$ are (equivalence classes of) multiplicative seminorms on $K[X,Y]$ that extend the absolute value on $K$, which we denote by $[\cdot]_z$.

Extending $\phi^\#$ to $\pberk$ by continuity, the Lyapunov exponent of $\phi$ is defined to be $$L(\phi):= \int_{\pberk} \log_v [\phi^\#]_z d\mu_\phi(z)\ , $$ where $\mu_\phi$ is the unique $\phi$-invariant probability measure satisfying $\phi^*\mu_\phi = d\cdot \mu_\phi$. We remark that when $K$ is non-Archimedean, $\mu_\phi$ need not be the measure of maximal entropy; see \cite{FRLErgodic} Section 5.2.

Okuyama has given quantitative approximations to $L(\phi)$ in terms of the multipliers of the $n$-periodic points of $\phi$ \cite{YO2}. He also gives a qualitative critera for approximating $L(\phi)$ with more general measures in \cite{YO} Lemma 3.1. The author has also given an approximation of $L(\phi)$ in terms of Rumely's crucial measures $\{\nu_{\phi^n}\}$ (see \cite{KJ2} Corollary 1). \\

Our main theorem is the following lower bound for Lyapunov exponents of rational maps defined over non-Archimedean fields:

\begin{thm}\label{thm:lowerboundonLyap}
Let $K$ be a complete, algebraically closed non-Archimedean field with $\textrm{char}(K) = 0$. Let $\phi\in K(T)$ be a rational map of degree $d$, and for $\gamma\in \PGL_2(K)$ let $\phi^\gamma := \gamma^{-1}\circ\phi\circ \gamma$.

Let $\kappa = \min(\log_v | m|\ : \ 1\leq m \leq d)$, noting that $\kappa \leq 0$. Let $\mathcal{L}_\phi$ be the Lipschitz constant for the action of $\phi$ on $\mathbb{P}^1(K)$ in the spherical metric. Then $$L(\phi) \geq \kappa- (d+1) \inf_{\gamma\in \PGL_2(K)}\log_v \mathcal{L}_{\phi^\gamma}\ .$$
\end{thm}

Note that, in the case that $\phi$ has potential good reduction, we can choose $\gamma\in \PGL_2(K)$ with $\mathcal{L}_{\phi^\gamma}=1$, so that we obtain the lower bound $L(\phi)\geq \kappa$. The map $\phi(z) = z^p$, where $p$ is the characteristic of the residue field, shows that this bound can be attained, however, it need not be sharp in all cases; for example, if $d>p$ is coprime to $p$ and $\phi(z) = z^d$, then again $\mathcal{L}_\phi = 1$, but we find $L(\phi) = 0 > \kappa$.

If we assume an integrability condition on the diameter of points in the support of $\mu_\phi$, we are able to remove the dependence on the Lipschitz constant:

\begin{thm}\label{thm:lowerboundonLyapbdd}
Let $K$ be a complete, algebraically closed non-Archimedean field, and suppose that $\textrm{char}(K)=0$. Let $\phi\in K(T)$ be a rational map of degree $d\geq 2$ such that $\log_v\diam_\infty(\cdot)\in L^1(\mu_\phi)$ and $\mathcal{J}(\phi)\subseteq \hberk$. Then $$L(\phi)\geq \kappa\ .$$
\end{thm}

The function $\log_v \diam_\infty(\cdot)$ will be defined in Section~\ref{sect:metrics} below. The set $\mathcal{J}(\phi)$ is the Julia set of $\phi$, which is the support of the measure $\mu_\phi$. The set $\hberk$ is the Berkovich hyperbolic line, which is the complement $\pberk \setminus \PP^1(K)$ (see Section~\ref{sect:Berk} below).

As an example, any flexible Latt\`es map satisfies the hypotheses of Theorem~\ref{thm:lowerboundonLyapbdd} (the Julia set for such a map is a segment in $\hberk$; see \cite{FRLErgodic} Section 5.1), as will any map with potential good reduction. We will show in Proposition~\ref{prop:goodredlyapzero} below that if the reduction of $\phi$ at the point $\zeta$ is also separable, then $L(\phi)=0$.\\

In the process of proving Theorem~\ref{thm:lowerboundonLyap}, we needed to strengthen the equidistribution of preimages of point masses $[\zeta]$, in the case that $\zeta\in \hberk$, to include test functions which have a logarithmic singularity at a point in $\PP^1(K)$ (see \cite{BRsmht}, \cite{CL}, \cite{FRL} for equidistribution using continuous test functions).  We show:

\begin{thm}\label{thm:logeq}
Let $K$ be a complete, algebraically closed non-Archimedean valued field, and let $\phi \in K(z)$ have degree $d\geq 2$. Fix a point $\xi\in \hberk$, and let $\nu_n = \frac{1}{d^n} (\phi^{n})^* [\xi]$, where $[\xi]$ is the Dirac mass at $\xi$. Fix a point $a\in \pberk$. We have $$ \int \log_v \delta(z, a)_{\zetaG} d\nu_n \to \int \log_v \delta(z, a)_{\zetaG} d \mu_\phi\ .$$
\end{thm}

Here, the pullback operator $\phi^*$ is defined in terms of multiplicity functions on $\pberk$; see Section~\ref{sect:pullbacks} below. The point $\zetaG\in \pberk$ is a distinguished point called the Gauss point (see Section~\ref{sect:Berk}). The Hsia kernel plays a fundamental role in potential theory on $\pberk$, and will be discussed in Section~\ref{sect:Berkpotential} below.

A quantitative version of Theorem~\ref{thm:logeq} is given in Proposition~\ref{prop:quantlogeq} below; it is worth noting that the error term is \emph{independent of the point $a$}, and depends only on the map $\phi$ and the point $\xi\in \hberk$.

\subsection{Outline}
In Section~\ref{sect:BGN}, we establish background on Berkovich space and dynamics on Berkovich space. Following this, in Section~\ref{sect:propLyap} we establish some basic properties of Lyapunov exponents, including the coordinate invariance of $L(\phi)$ and a formulation of $L(\phi)$ in terms of the standard derivative $\phi'$. In Section~\ref{sect:mainproof} we recall a notion of distortion introduced in \cite{BIJL} and use it to establish the lower bound $$\log_v [\phi']_\zeta \geq \kappa + \log_v\diam_\infty(\phi(\zeta))-\log_v\diam_\infty(\zeta)\ ,$$ which we note is only valid for $\zeta\in \hberk$. The idea of the proof of the main theorems is to integrate both sides of this inequality against $\mu_\phi$; this can be done directly under the hypotheses of Theorem~\ref{thm:lowerboundonLyapbdd}, but not in general. For the general case, we rely on some equidistribution-type results for the measures $\frac{1}{d^n}(\phi^{n})^*[\zetaG]$; these are given in Section~\ref{sect:eq}. Finally, we establish Theorem~\ref{thm:lowerboundonLyap} in Section~\ref{sect:generalcase}. 

\subsection{Acknowledgements}
The author would like to thank Robert Rumely, Rob Benedetto, Y\^{u}suke Okuyama, and Laura DeMarco for their assistance and helpful conversations, as well as anonymous referees of earlier versions of this paper whose feedback helped clarify the exposition and simplify several proofs. The author gratefully acknowledges support from the NSF grant DMS-1344994 of the RTG in Algebra, Algebraic Geometry, and Number Theory, at the University of Georgia.

\section{Background and Notation}\label{sect:BGN}

\subsection{Spherical Metric and Spherical Derivatives}

The projective line $\PP^1(K)$ over $K$ is the quotient of $K^2\setminus\{(0,0)\}$ by the scaling action of $K^\times$. Let $\vz=(z_1, z_2), \vw=(w_1, w_2)\in K^2\setminus \{(0,0)\}$; the spherical distance between $\vz$ and $\vw$ is $$||\vz, \vw|| = \frac{|z_1w_2-z_2w_1|_v}{\max(|z_1|_v, |z_2|_v)\cdot \max(|w_1|_v, |w_2|_v)}\ .$$ This metric is scaling invariant, and hence descends to a well-defined function on $\PP^1(K)$. For points $\tilde{z}, \tilde{w}\in K$, we further have $$||\tilde{z},\tilde{w}|| = \frac{|\tilde{z}-\tilde{w}|_v}{\max(1, |\tilde{z}|_v)\cdot\max(1, |\tilde{w}|_v)}\ , $$ where $\tilde{z}, \tilde{w}$ are projections of $\vz, \vw$ (resp.). 

Let $\phi\in K(z)$ be a rational map (of any degree), viewed as an endomorphism of $\PP^1(K)$. The spherical derivative of $\phi$ is the derivative of $\phi$ with respect to the spherical metric: $$\phi^\#([z,w]) = \lim_{[u,v]\to[z,w]} \frac{||\phi([u,v]), \phi([z,w])||}{||[u,v], [z,w]||}\ .$$ Using the expression for the spherical derivative in the affine plane $K$ we find (for $z\in K$ not a pole of $\phi$) 
\begin{align}
\phi^\#(z) &= \lim_{w\to z} \frac{||\phi(z), \phi(w)||}{||z, w||}\nonumber \\
& = \lim_{w\to z} \left(\frac{|\phi(z) - \phi(w)|_v}{|z-w|_v} \cdot \frac{||z, w||}{||\phi(z), \phi(w)||}\right)\nonumber\\
& = |\phi'(z)|_v \frac{\max(1, |z|_v)^2}{\max(1, |\phi(z)|_v)^2}\label{eq:sphericalderivative}\ .
\end{align}

With this, we see that the spherical derivative satisfies the chain rule:

\begin{lemma}\label{lem:sphericalchainrule}
Let $\phi, \psi\in K(z)$. For $z\in K\setminus \{\textrm{poles of }\phi\circ\psi, \textrm{ poles of }\psi\}$, the spherical derivative satisfies the following version of the chain rule:
\[
(\phi\circ\psi)^\#(z) = \phi^\#(\psi(z))\cdot \psi^\#(z)\ .
\]
\end{lemma}

\begin{proof}
For $z\in K\setminus \{\textrm{poles of }\phi\circ\psi,\textrm{ poles of }\psi\}$, the expression in (\ref{eq:sphericalderivative}) and the usual chain rule give
\begin{align*}
(\phi\circ\psi)^\#(z) &= \frac{\max(1, |z|_v)^2}{\max(1, |\phi\circ\psi(z)|_v^2)} \cdot |(\phi\circ\psi)'(z)|_v\\
& = \frac{\max(1, |\psi(z)|_v)^2}{\max(1, |\phi(\psi(z))|_v)^2}\cdot \frac{\max(1, |z|_v)^2}{\max(1, |\psi(z)|_v)^2}\cdot |\phi'(\psi(z))|_v\cdot |\psi'(z)|_v\\
& = \left(\frac{\max(1, |\psi(z)|_v)^2}{\max(1, |\phi(\psi(z))|_v)^2}\cdot |\phi'(\psi(z))|_v\right)\cdot \left(\frac{\max(1, |z|_v)^2}{\max(1, |\psi(z)|_v)^2} \cdot |\psi'(z)|_v\right)\\
& = \phi^\#(\psi(z))\cdot \psi^\#(z)\ .
\end{align*}
\end{proof}

\subsection{The Berkovich Projective Line}\label{sect:Berk}
For a thorough development of the Berkovich projective line, the reader is directed to the original work of Berkovich \cite{Ber}, or to the book of Baker and Rumely \cite{BR}, particularly Chapters 1 and 2. We briefly recall the construction of $\pberk$ here, following the notation of \cite{BR}.

Let $K$ be a complete, algebraically closed non-Archimedean valued field. In the topology induced by the absolute value $|\cdot|_v$ on $K$, the projective line $\PP^1(K)$ is totally disconnected and so not a suitable space for analysis or potential theory. One works instead with the Berkovich projective line $\pberk$, a uniquely path connected, compact, Hausdorff space that contains $\PP^1(K)$ as a dense subset. 

Formally, the Berkovich affine line $\aberk$ is the collection of multiplicative seminorms on $K[T]$ which extend the absolute value in $K$.  Points $[\cdot]_z\in \aberk$ correspond to (cofinal equivalence classes of) nested decreasing sequences of closed discs $D(a_i, r_i) \supset D(a_{i+1}, r_{i+1})\supset \cdots$ by the identification $$[f]_z = \lim_{i\to \infty}\  \sup_{w\in D(a_i, r_i)} |f(w)|_v \ .$$ If the intersection of the discs $D(a_i, r_i)$ is a single point $\{a\}\subseteq K$, we call $[\cdot]_z$ a point of type I; this gives a natural embedding of $K$ into $\aberk$. If the intersection of the $D(a_i, r_i)$ is again a disc $D(a,r)$, we call $[\cdot]_z$ a point of type II or of type III, corresponding to whether $r\in |K^\times|$ or $r\not\in |K^\times|$. If the intersection of the discs $D(a_i, r_i)$ is empty, but $\lim_{i\to\infty} r_i >0$, then $[\cdot]_z$ is called a point of type IV. Such points do not occur when the field $K$ is spherically complete. We will often denote points of type I, II, and III by $\zeta_{a,r}$ for shorthand, where, as above $D(a,r) = \bigcap D(a_i, r_i)$. The type II point corresponding to the unit disc is called the Gauss point, and will be denoted $\zetaG = \zeta_{0,1}$. \\

The weak topology on $\aberk$ is the coarsest topology such that the map $z\mapsto [f]_z$ on $\aberk$ is continuous for all $f\in K[T]$. Below we will give an explicit (sub)basis for this topology. In the weak topology, points of type I are dense in $\aberk$; consequently, the action of a nonconstant rational map $\phi$ on $K$ can be extended continuously to all of $\aberk$. Heuristically, for a point $[\cdot]_z\in \aberk$, one has $[f]_{\phi(z)} = [f\circ \phi]_z$, though there are some technical points that must be addressed to make this precise; see \cite{BR} Section 2.3. It is important to note that the action of a rational map preserves the types of points in $\aberk$. We also have that points of type II are also dense in the weak topology, as are points of type III.

One obtains the Berkovich projective line $\pberk$ by glueing two copies of $\aberk$ via the involution $\theta(z) = \frac{1}{z}$; see \cite{BR} Section 2.2. Equivalently, $\pberk$ can be obtained as the one-point compactification of $\aberk$ by adding a single type I point $\infty$. In the weak topology, $\pberk$ is compact and Hausdorff, but in general is not metrisable. The Berkovich hyperbolic line is $\hberk = \pberk \setminus \PP^1(K)$.

The automorphism group of $\pberk$ is the set of Mobius transformations $\PGL_2(K)$. Such maps act transitively on type II points in $\hberk$, with stabilizer isomorphic to $K^\times \PGL_2(\mathcal{O})$ (see \cite{Ru1}, Proposition 1.1). Additionally, maps in $\PGL_2(K)$ act transitively on arcs $[a,b]$ with type I endpoints; see \cite{BR} Corollary 2.13.

\subsubsection{Tree Structure}\label{sect:tree}
A collection of discs $\{\zeta_{a,r}\}_{r\in [s,t]}$ gives an embedding of the real interval $[s,t]$ into $\pberk$ and allows one to endow $\pberk$ with the structure of a tree. The endpoints of this tree are points of type I or of type IV. Each point of type II is a branch point, and the branches are in one-to-one correspondence with elements of $\PP^1(k)$. To see this, note that the unit disc can be written as $$D(0,1) = \bigsqcup_{a\in k} b_a + D(0,1)^-\ ,$$ where $b_a\in \mathcal{O}$ satisfies $\tilde{b}_a = a$. The rays $\{\zeta_{b_a, r}\}_{r\in [0,1]}$ all terminate at $\zetaG$ and hence form branches of the tree at $\zetaG$. The branch towards $\infty$ corresponds to the segment $\{\zeta_{0, R}\}_{R\in [1, \infty)}$. Points of type III are not branch points.

Given a point $\zeta\in \pberk$, we define the tangent space $T_\zeta$ at $\zeta$ to be the collection of equivalence classes of paths emanating from $\zeta$. For points of type I and type IV, $T_\zeta$ consists of a single equivalence class consisting of the paths leading into $\hberk$. For points of type II, $T_\zeta$ is in one-to-one correspondence with $\PP^1(k)$; this is seen by considering $\zeta = \zetaG$ and applying the remarks in the previous paragraph. For points of type III, $T_\zeta$ has two equivalence classes. We will most often call elements of $T_\zeta$ `directions', and denote them by $\vv, \vw$, etc.

Alternatively, consider the connected components of $\pberk \setminus \{\zeta\}$. These are in one-to-one correspondence with the tangent directions at $\zeta$, and we denote each such component by $B_{\vv}(\zeta)^-$, where $\vv\in T_\zeta$ (in \cite{BR}, this is denoted by $B_\zeta(\vv)^-$). The collection of all such components as $\zeta$ varies in $\pberk $and $\vv$ varies in $T_\zeta$ gives a subbasis for the weak topology on $\pberk$. 

A subtree $\Gamma\subseteq \hberk$ will be a closed, connected subgraph of $\hberk$ such that $\Gamma$ has finitely many edges. Given a point $P\in \Gamma$, the tangent space $T_P(\Gamma)$ is the collection of equivalence classes of paths emanating from $P$ that intersect $\Gamma$. By assumption, there are necessarily finitely many such directions. The valence $v_\Gamma(P)$ is the cardinality of $T_P(\Gamma)$. 

The tree structure on $\pberk$ also allows one to define a natural retraction map $r_{V,U}$ between a set $V\subseteq \pberk$ and a closed, connected subset $U\subseteq \pberk$: any point in $U$ is fixed by $r_{V,U}$. Fix a point $x\in U$, and for any $y\in V\setminus U$ define $r_{V,U}(y) $ to be the first point on the segment $[y, x]$ that intersects $U$. By the unique path connectedness of $\pberk$, this map is well-defined and independent of our choice of $x\in U$. If $V=\pberk$, we may sometimes abbreviate this map $r_{U}$. 

\subsubsection{Metrics on $\pberk$}\label{sect:metrics}

The tree structure also implies that $\pberk$ is uniquely path connected. Fix a point $\zeta\in \pberk$. We define the join of two points $x,y\in \pberk$ relative to $\zeta$, denoted $x\wedge_\zeta y$, to be the first point where the segments $[x,\zeta]$ and $[y, \zeta]$ intersect. 

The `big metric' on $\hberk$ is defined as follows: given two points $x,y\in \hberk$, let $$\rho(x,y) = 2\log_v \diam_{\infty}(x\wedge_\infty y) - \log_v \diam_\infty(x) - \log_v \diam_\infty(y)\ ;$$ here, $\diam_\infty(x):=\lim_{i\to\infty} r_i$, where $\{D(a_i, r_i)\}$ is a representative of the class of discs corresponding to $x$. For points of type II or of type III, it is simply the diameter of the underlying disc. We extend this metric to all of $\pberk$ by saying $\rho(x,y) = \infty$ if either $x$ or $y$ is in $\PP^1(K)$. The metric topology defined by $\rho$ is called the strong topology. It is strictly finer than the weak topology on $\pberk$. 

\subsection{Potential Theory on $\pberk$}\label{sect:Berkpotential}
In this section, we develop some potential theory on the Berkovich projective line $\pberk$ over $K$. A rigorous development of this subject can also be found in Chapters 3-9 of the book \cite{BR}. A development of potential theory on more general Berkovich curves has also been carried out in the thesis of Thuillier \cite{Thu}.

\subsubsection{The Fundamental Potential Kernel and the Hsia Kernel}

For a fixed $z\in \pberk$, one defines a potential kernel on $\pberk$, denoted $\langle x,y\rangle_z$, as follows: consider the paths $[x,z]$, $[y, z]$, and let $w = x\wedge_\zeta y$ be the first point where these paths intersect. We define $$\langle x,y\rangle_z = \rho(w, z)\ .$$ 

For fixed $y$, the function $\langle \cdot, y\rangle_z$ is increasing with a constant slope 1 along the segment $[z,y]$, and is constant on branches off of this segment. More generally, $\langle x,y\rangle_z$ is non-negative, symmetric in $x$ and $y$, and jointly continuous (for the \emph{strong} topology) in the variables $x,y$, and $z$ (see \cite{BR} Proposition 3.3 -- our $\langle x,y\rangle_z$ is their $j_z(x,y)$). 

We can use the potential kernel above to define an extension of the classical spherical distance to $\pberk$ by letting $$-\log_v ||x,y|| = \langle x,y\rangle_{\zetaG}\ .$$ This definition can be generalized to more arbitrary kernels called Hsia kernels. For a fixed $\zeta\in \pberk$, the generalized Hsia kernel at $\zeta$ is defined to be $$\delta(x,y)_\zeta = \frac{||x,y||}{||x,\zeta||\cdot ||y, \zeta||}\ ,$$for all $x,y\in \pberk \setminus \{\zeta\}$. 

We remark that when $\zeta = \infty$ and $x,y$ are points of type I, then $\delta(x,y)_\infty$ is the quantity $\diam_\infty(x\wedge_\infty y)$ appearing in the definition of the big metric above. More generally, we can use the Hsia kernel to define a diameter relative to any fixed point $\zeta\in \pberk$ by $$\diam_{\zeta}(x) = \delta(x,x)_\zeta\ .$$ This admits a decomposition as (see \cite{BR} Equation 4.32) $$\diam_{\zeta}(x) = \frac{||x,x||}{||x, \zeta||^2}\ .$$

The `small metric', defined on all of $\pberk$, can be defined in terms of the diameter functions introduced above. For $x,y\in \pberk$, we have $$\textrm{d}_{\pberk}(x,y) = 2\diam_{\zetaG}(x\wedge_{\zetaG} y) - \diam_{\zetaG}(x) - \diam_{\zetaG}(y)\ .$$ This function is an extension of (twice) the classical spherical distance on $\PP^1(K)$. The small metric $\textrm{d}_{\pberk}$ also generates the strong topology, since $\rho$ and $\textrm{d}_{\pberk}$ are locally bounded in terms of one another. The action of $\phi$ is Lipschitz continuous with respect to the small metric (see \cite{BR} Proposition 9.37); explicit bounds on the Lipschitz constant are given in \cite{RW}.

\subsubsection{The Laplacian on $\pberk$}

Fix a subtree $\Gamma\subseteq \hberk$, and let $f:\Gamma\to \RR$. Given a point $P\in \Gamma$, the derivative of $f$ in the direction $\vv\in T_P(\Gamma)$, if it exists, is the limit $$\partial_{\vv}(f)(P) := \lim_{t\to 0} \frac{ f(P+t\vv)-f(P) }{t}\ .$$ A function $f:\Gamma\to \RR$ is said to be continuous and piecewise affine if $f$ is continuous on $\Gamma$ and there is a finite set $\mathcal{S}\subseteq \Gamma$ which includes the branch points of $\Gamma$ and such that the restriction of $f$ to any segment of $\Gamma\setminus \mathcal{S}$ is affine. We denote the space of all such functions $\CPA(\Gamma)$.

The Laplacian of a function $f\in \CPA(\Gamma)$ is defined to be the measure $$\Delta_\Gamma f = -\sum_{P\in \Gamma} \sum_{\vv\in T_P(\Gamma)} \partial_{\vv}(f)(P) [P]\ .$$ Here $[P]$ is the Dirac mass at $P$. This definition can be extended to a broader class of functions, called functions of bounded differential variation and denoted $\BDV(\Gamma)$. Essentially, these are functions which `do not oscillate too much' on $\Gamma$. 

We will use the fact that the Laplacian satisfies the following property (see \cite{BR} Proposition 3.14): for any $z\in \Gamma$ and any $f,g\in \BDV(\Gamma)$, $$\int_\Gamma f d\Delta_\Gamma(g) = \int_\Gamma g d\Delta_\Gamma(f) = \iint_{\Gamma\times \Gamma} \langle x,y\rangle_z d\Delta_\Gamma(f)(x) d\Delta_\Gamma(g)(y)\ .$$

The definition of the Laplacian can be extended to subdomains $U\subseteq \pberk$ by means of limits of coherent measures. Given a subdomain $U\subseteq \pberk$ and an exhaustion $\{\Gamma_i\}_{i\in I}$ of $U$ (here, $I$ is any directed set), we say that a family of measures $\{u_{\Gamma_i}\}_{i\in I}$ is coherent if the total masses are uniformly bounded $|u_\Gamma|(\Gamma) \leq B$ by a constant $B$ independent of $\Gamma$ and for any pair of graphs $\Gamma_1 \subseteq \Gamma_2$, we have $$u_{\Gamma_1}(A) = u_{\Gamma_2}(r_{\Gamma_2, \Gamma_1}^{-1}(A))$$ for each Borel subset $A\subseteq \Gamma_1$. One can show (see \cite{BR} Proposition 5.10) that coherent measures can be glued together to define a measure $\mu$ on $U$ that satisfies $\mu_{\Gamma_i}(A) = \mu (r_{U, \Gamma_i}^{-1}(A))$ for any $\Gamma_i$ and any Borel set $A\subseteq \Gamma_i$. 

After suitably defining a space $\BDV(U)$ of functions of bounded differential variation on $U$ (see \cite{BR} Section 5.4), the Laplacian on $U$ is defined to be $$\Delta_U(f) = \lim_{\substack{\to \\ \Gamma}} \Delta_\Gamma (f)\ .$$

We have chosen to use this construction of the Laplacian because the notion of a coherent measure will be used in the proof of Proposition~\ref{prop:wildintegrallowerbound} below. Alternative constructions of the Laplacian on $\pberk$ have been given in the work of Favre, Jonsson, and Rivera-Letelier (\cite{FRL} and \cite{FRLErgodic}), and in the work of Thuillier (\cite{Thu}). A comparison of these Laplacians is carried out in \cite{BR} Section 5.8; for our purposes, it is important to note that the Laplacian described here is the \emph{negative} of the Laplacian constructed by Favre, Jonsson, and Rivera-Letelier. 

\subsubsection{Potential Functions}\label{sect:potentialfunctions}
We will denote by $\mathcal{M}$ the space of all finite, signed Borel measures on $\pberk$. It will be important later to note that these measures are also Radon measures on $\pberk$ (\cite{BR} Lemma 5.6). 

Fix a base point $\zeta_0\in \hberk$. Given a finite signed Borel measure $\lambda\in \mathcal{M}$, one can show that $$u_\lambda(z, \zeta_0) := -\int_{\pberk} \log_v\delta(w, z)_{\zeta_0} d\lambda(w)$$ is a potential function for $\lambda-[\zeta_0]$ in the sense that $\Delta u_\lambda(\cdot, \zeta_0) = \lambda-[\zeta_0]$ (\cite{BR} Example 5.22). This differs slightly from the potential functions defined in \cite{FRLErgodic} Section 4.1 by an additive constant and by a negative sign. As a special case, when $\lambda=[z]$ for any point $z\in \pberk$, we have $$\Delta u_\lambda(\cdot, \zeta_0) = \Delta\langle \cdot, z\rangle_{\zeta_0} = [z] - [\zeta_0]\ .$$

We say that $\lambda$ has continuous (resp. bounded) potentials if and only if for some $\zeta_0\in \hberk$, the function $u_\lambda(\cdot, \zeta_0)$ is continuous (resp. bounded) on $\pberk$. Using the change of variables formula for the Hsia kernel (see \cite{BR} Equation (4.29)), this property is independent of the basepoint $\zeta_0$. 

\subsubsection{Multiplicities and Imbalance Formulas}

In this section, we briefly describe a notion of multiplicity on $\pberk$ which extends the usual notion of multiplicity on $\PP^1(K)$. These multiplicities are used to give a lower bound for a distortion constant $\hat{\delta}(\phi, \zeta)$. Lower bounds of this sort were used in an essential way in \cite{BIJL}.

There are a number of equivalent approaches that one can take to defining multiplicities on the Berkovich line. We begin by defining the directional multiplicity of $\phi$ at a point in $\pberk$. Fix $P\in \pberk$, and let $\vv\in T_P$. There exists an integer $m\in \{1, 2, ..., d\}$ such that, for any $y\in B_{\vv}(P)^-$ sufficiently close to $P$, we have $$\rho(\phi(x), \phi(y)) = m\cdot \rho(x,y)\ .$$ We call $m=m_\phi(P, \vv)$ the directional multiplicity of $\phi$ at $P$ in the direction $\vv$. We define the multiplicity $m_\phi(P)$ as follows (see \cite{BR} Theorem 9.22 (C)) : fix any direction $\vw\in T_{\phi(P)}$ at the image $\phi(P)$. Then $$m_\phi(P) = \sum_{\substack{\vv\in T_P \\ \phi_* \vv = \vw}} m_\phi(P, \vv)\ .$$ The notion of multiplicity given here is consistent with the usual topological notion of multiplicity in terms of counting preimages (see \cite{BR} Corollary 9.17). 

Let $\zeta=\zeta_{a,r}\in \pberk$ be a point of type II, and let $\xi=\phi(\zeta)$. Choose $\alpha, \beta\in \pberk$ to lie in different connected components of $\pberk \setminus \{\xi\}$, and let $\alpha_1, ..., \alpha_d, \beta_1, ..., \beta_d$ to be their respective preimages, listed with multiplicity. Fix any $\vv\in T_\zeta$. Then $$N_{\vv}^-(\phi, \zeta_{a,r}, \alpha) = \# \{\alpha_1, ..., \alpha_n\} \cap B_{\vv}(\zeta)^-$$ counts the number of preimages of $\alpha$ lying in $B_{\vv}(\zeta)^-$.

We finally introduce the surplus multiplicity as a means of counting solutions; a good reference on surplus multiplicity is the paper of Xander Faber \cite{XF}. Fix a point $\zeta\in \hberk$, and let $\vv\in T_\zeta$. The image $\phi(B_{\vv}(\zeta)^-)$ will be either the ball $B_{\phi_* \vv}(\phi(\zeta))^-$, or else all of $\pberk$. In terms of points of type I, this implies that for a given disc $D(a,r)$, either $\phi(D(a,r))$ is again a disc $D(\phi(a), s)$ or else is the entire projective line $\PP^1(K)$. More precisely, there is an integer $s$ depending only on $\zeta$ and $\vv$ such that, for any $y\in \PP^1(K)$, $$\#\phi^{-1}(y) \cap B_{\vv}(\zeta)^- =\left\{\begin{matrix} m_\phi(\zeta, \vv) + s, & y \in B_{\phi_* \vv}(\phi(\zeta))^-\\ s, & y\not\in B_{\phi_* \vv}(\phi(\zeta))^-\end{matrix}\right. \ .$$ We write $s_\phi(\zeta, \vv) = s$ for this integer, which is called the surplus multiplicity. It is an integer in the range $\{0,1, 2, ..., d-1\}$. 

\subsection{Pullbacks of Maps and Measures}\label{sect:pullbacks}

With the notion of multiplicities introduced above, one can define a notion of pullback by a rational map $\phi$ both for functions and for measures. Let $\phi \in K(z)$ have degree $d\geq 2$. For any continuous, real valued function $g$ on $\pberk$, we define the pullback of $g$ under $\phi$ to be the real-valued function $$\phi_* g(x) = \sum_{\phi(y) = x} m_\phi(y) g(y)\ .$$ This is also known as the Frobenius-Perron or Transfer operator.
Given a measure $\lambda\in \mathcal{M}$, we define the pullback of $\lambda$ under $\phi$ to be the unique Radon measure $\phi^* \lambda$ satisfying $$\int g\ d\phi^* \lambda = \int \phi_* g\ d\lambda\ .$$ The existence and uniqueness of the measure $\phi^*\lambda$ is guaranteed by the Riesz representation theorem (see \cite{Fol} Theorem 7.17). As a special case, fix a point $z\in \pberk$ and let $\zeta_1, ..., \zeta_d$ be the pre-images of $z$ (listed with multiplicity). Then $\phi^* [z] = \sum_{k=1}^d [\zeta_i]$ is the measure supported at the preimages of $z$ weighted with multiplicity.

We will often use the fact that the pullback of the Laplacian satisfies $$\phi^* \Delta (f) = \Delta(f\circ \phi)$$ for any $f\in \BDV(\pberk)$ (see \cite{BR} Proposition 9.56). 

We will also make use of the pushforward operator for a measure $\lambda\in \mathcal{M}$: if $E$ is a measurable subset of $\pberk$, then $\phi_*\lambda(E) = \lambda(\phi^{-1}(E))$. Equivalently, for any continuous function $\psi$ on $\pberk$, the pushforward is the unique measure satisfying
\[
\int \psi\ d(\phi_*\lambda) = \int \psi\circ \phi\ d\lambda\ .
\]

\subsection{Dynamics on the Berkovich Line}

In this section, we briefly recall the construction of the equilibrium measure $\mu_\phi$ attached to a rational map $\phi$ acting on $\pberk$. 

\subsubsection{The Equilibrium Measure}
In this section, we will follow the approach of Favre and Rivera-Letelier (see \cite{FRLErgodic}, Proposition-D\'efinition 3.1) to construct the equilibrium measure. This construction will again be used in several of the proofs below.

A point $\zeta\in \hberk$ is said to be exceptional if the backwards orbit of $\zeta$ (under $\phi$) is finite. For any non-exceptional point $\zeta\in\hberk$, the function  $g_1(\cdot, \zeta) = \frac{1}{d}\sum_{\phi(\zeta_i) = \zeta} \langle \cdot, \zeta_i\rangle_\zeta$ is a potential for the measure $\frac{1}{d} \phi^* [\zeta] - [\zeta]$. By a telescoping series argument using the pullback formula for the Laplacian, one finds $$\frac{1}{d^n}(\phi^{n})^* [\zeta] = [\zeta] + \Delta g_n\ ,$$ where $$g_n(\cdot, \zeta) = \sum_{k=0}^{n-1} \frac{g_1(\phi^k(\cdot), \zeta)}{d^k}\ .$$ Writing $\nu_n = \frac{1}{d^n} (\phi^n)^* [\zeta]$, one can readily check that these measures satisfy $\phi_* \nu_n = \nu_{n-1}$. 

The functions $g_n(\cdot, \zeta)$ converge uniformly on $\pberk$ to a function $g_\infty(\cdot, \zeta)$, and hence the measures $\frac{1}{d^n}(\phi^{n})^* [\zeta]$ converge weakly\footnote{One can also obtain equidistribution for pullbacks of non-exceptional type I points if $K$ is the completion of the algebraic closure of the completion of a number field at a finite place; see \cite{BRsmht} Theorem 2.3 and \cite{FRL} Th\'eor\`eme 2.} to a measure $\mu_\phi$ called the equilibrium measure of $\phi$ (\cite{FRLErgodic}, Proposition-D\'efinition 3.1). It is a probability measure that independent of the choice of $\zeta$, and which satisfies $$\phi_* \mu_\phi = \mu_\phi\ , \ \phi^* \mu_\phi = d\cdot \mu_\phi\ .$$ The measure $\mu_\phi$ charges points if and only if $\phi$ has good reduction; in this case, there is a unique point receiving mass which corresponds to the conjugate attaining potential good reduction. In particular, $\mu_\phi$ never charges type I points (see, e.g. \cite{BR} Corollary 10.47).

We remark that, in contrast to the situation over $\CC$, the measure $\mu_\phi$ is \emph{not}, in general, the measure of maximal entropy for $\phi$ (\cite{FRLErgodic}, Section 5.2). Nevertheless, it is mixing with respect to $\phi$, it does not charge points of $\PP^1(K)$ (\cite{FRLErgodic}, Th\'eor\`eme A). The support of $\mu_\phi$ is called the Berkovich Julia set of $\phi$, and will be denoted $\mathcal{J}(\phi)$; its complement is the Fatou set $\mathcal{F}(\phi)$ (\cite{BR} Chapter 10)

\section{Properties and Estimates on $L_v(\phi)$:}\label{sect:propLyap}
Here we establish some of the basic properties of $L_v(\phi)$. 
\subsection{An Alternate Formula for $L_v(\phi)$}
Since $\log_v[\phi^\#]\in L^1(\mu_\phi)$ (e.g., equation (1.2) in \cite{YO} is finite), the Birkhoff ergodic theorem tells us that the Lyapunov exponent $L_v(\phi)$ can be computed $$L_v(\phi) =\lim_{n\to\infty} \frac{1}{n} \sum_{k=0}^n \log_v [\phi^\#]_{\phi^k(z)} $$ for $\mu_\phi$-almost every $z\in \pberk$. Given such a $z$, we let $$\overline{L}_v(\phi)(z) := \lim_{n\to\infty} \frac{1}{n} \sum_{k=0}^n \log_v [\phi^\#]_{\phi^k(z)}\ .$$Similarly, note that $\log_v[\phi']_z\in L^1(\mu_\phi)$: this follows from the fact that $\log_v[\phi']_z$ can be written as a finite linear combination of potentials $\langle z, a_i\rangle_{\zeta}$ for a fixed basepoint $\zeta\in \hberk$ and appropriate $a_i\in \mathbb{P}^1(K)$, and the functions $\langle \cdot, a_i\rangle_\zeta$ are in $L^1(\mu_\phi)$ (see, e.g., \cite{FRLErgodic} Lemma 4.3). Letting $$\hat{L}_v(\phi) = \int_{\pberk} \log_v [\phi']_z d\mu_\phi\ ,$$ the Birkhoff ergodic theorem tells us that $$\hat{L}_{v}(\phi) =  \lim_{n\to\infty} \frac{1}{n}\sum_{k=0}^{n} \log_v[\phi']_{\phi^k(z)}\  $$ for $\mu_\phi$ almost every $z$. Given such a $z$, we define $$\overline{\hat{L}}_v(\phi)(z) = \lim_{n\to\infty} \frac{1}{n}\sum_{k=0}^{n} \log_v[\phi']_{\phi^k(z)}\ .$$

Our goal in this section is to use the two time-averages $\overline{L}_v, \overline{\hat{L}}_v$ to show that the two space averages $L_v, \hat{L}_v$ are equal. To do this, we begin with two lemmas.

\begin{lemma}\label{lem:coordchangeok}
The Lyapunov exponent is coordinate invariant; that is, for any $\gamma\in \PGL_2(K)$, we have $$L_v(\phi) = L_v(\phi^\gamma)\ .$$
\end{lemma}

\begin{proof}
The proof is an application of the chain rule and the change of variables formula for integrals. Recall from (\ref{eq:sphericalderivative}) that for any rational map $\phi(z)$ of degree $d$ and any $z\in K\setminus \{\textrm{poles of }\phi\}$, we have \begin{align}\label{eq:gsharpgprime}\phi^\#(z) = \frac{\max(1, |z|_v)^2}{\max(1, |\phi(z)|_v)^2} |\phi'(z)|_v\ .\end{align}
In particular, for $\gamma \in \PGL_2(K)$ we have \begin{align*}(\gamma^{-1})^\#(\gamma(z)) &= \frac{\max(1, |\gamma(z)|_v)^2}{\max(1, |z|_v)^2}\cdot |(\gamma^{-1})'(\gamma(z))|_v\\ & =  \frac{\max(1, |\gamma(z)|_v)^2}{\max(1, |z|_v)^2}\cdot \frac{1}{|\gamma(z)|_v}\\ & = \frac{1}{\gamma^\#(z)}\ .\end{align*} Note that by continuity this extends to $\pberk$, where it can be formulated as \begin{align}\label{eq:IFTchordal} [(\gamma^{-1})^\#\circ\gamma]_z \cdot [\gamma^\#]_z = 1\ .\end{align}

We now show that $\mu_{\phi^\gamma} = \gamma^{-1}_*\mu_{\phi}$. Recall that the invariant measure can be obtained as an averaged pullback of a fixed point mass (see \cite{FRLErgodic} Proposition-D\'efinition 3.1). More precisely, for a point $P\in \hberk$ that is non-exceptional for $\phi^\gamma$, we have $$\mu_{\phi^\gamma} = \lim_{n\to\infty} \frac{1}{d^n} (\phi^\gamma)^{n*}[P]\ ;$$ unravelling the definition of the pullback formula, we have

\begin{align*}
(\phi^\gamma)^*[P] & = \frac{1}{d} \sum_{\gamma^{-1} \circ \phi \circ \gamma (S) = P} m_{\gamma^{-1}\circ \phi\circ \gamma} (S) [S]\\
& = \frac{1}{d} \sum_{\phi(\gamma(S)) = \gamma(P)} m_{\phi}(\gamma(S)) [ \gamma^{-1} (\gamma(S))]\\
& = (\gamma^{-1})_*\left(\phi^* [\gamma(P)]\right)\ .
\end{align*} Generalizing this to the iterates of $\phi$ and passing to the limit, we obtain \begin{align}\label{eq:measurecomp}\mu_{\phi^\gamma} = (\gamma^{-1})_*\ \mu_{\phi,}\ .\end{align}

We finish the proof as follows. By (\ref{eq:sphericalderivative}) and the chain rule for the spherical derivative (see Lemma~\ref{lem:sphericalchainrule} above), we find $$(\gamma^{-1}\circ\phi\circ\gamma)^\#(z) = (\gamma^{-1})^\#(\phi(\gamma(z))) \cdot \phi^\#(\gamma(z)) \cdot \gamma^\#(z)\ $$ away from a finite set of points. Inserting this into the definition of the Lyapunov exponent gives 
\begin{equation}\label{eq:lyapchaindecomp}
L_v(\phi^\gamma) = \int_{\pberk}\log_v [(\gamma^{-1})^\#\circ\phi\circ\gamma]_zd\mu_{\phi^\gamma}(z) + \int_{\pberk} \log_v[\phi^\#\circ\gamma]_z d\mu_{\phi^\gamma}(z)\nonumber + \int_{\pberk} \log_v[\gamma^\#]_zd\mu_{\phi^\gamma}(z) \ .
\end{equation} Applying (\ref{eq:measurecomp}) to each of the integrals, we have

\begin{equation}\label{eq:lyapchaindecompmod}
L_v(\phi^\gamma) = \int_{\pberk}\log_v [(\gamma^{-1})^\#\circ\phi]_zd\mu_{\phi}(z) + \int_{\pberk} \log_v[\phi^\#]_z d\mu_{\phi}(z)+ \int_{\pberk} \log_v[\gamma^\#\circ \gamma^{-1}]_zd\mu_{\phi}(z) \ .
\end{equation} Here, the second term is precisely $L_v(\phi)$. By the $\phi$-invariance of $\mu_\phi$ and the relation in (\ref{eq:IFTchordal}), the first and third integrals sum to 0. 
\end{proof}

\begin{lemma}\label{lem:otherLyapcoordinvt}
For each $v\in \mathcal{M}_K$, the exponent $\hat{L}_v(\phi)$ is coordinate invariant; that is, for any $\gamma\in \PGL_2(K_v)$, we have $$\hat{L}_v(\phi) = \hat{L}_v(\phi^\gamma)\ .$$
\end{lemma}

\begin{proof}
The proof here is essentially identical to that of the previous lemma; expression (\ref{eq:IFTchordal}) holds analogously for the standard derivative by the inverse function theorem, and (\ref{eq:measurecomp}) is independent of the derivative of $\phi$.
\end{proof}

We now prove that the Lyapunov exponent defined in terms of the chordal derivative is in fact equal to the Lyapunov exponent defined in terms of the standard derivative. 

\begin{prop}\label{prop:equalLyapunovs}
Let $\phi\in K(z)$ have degree $d\geq 2$. Then $L_v(\phi) = \hat{L}_v(\phi)$.
\end{prop}

\begin{proof}
Fix a coordinate system so that $\infty$ is in the Fatou set for $\phi$; over a non-Archimedean field $K$, the Fatou set always contains type I points (\cite{Ben}, Corollary 1.3), so that this is always possible. By the above lemmas, both $L_v(\phi)$ and $\hat{L}_v(\phi)$ are unaffected by this change of coordinates. 

We may choose $\zeta_{0,R}$ for $R$ sufficiently large so that the Julia set of $\phi$ lies in $V=\pberk \setminus B_{\vv_\infty}(\zeta_{0,R})$.
By the Birkhoff ergodic theorem, we can find $z_*\in \supp(\mu_\phi)$ with $$L_v(\phi) = \overline{L}_v(\phi)(z_*)\ ;\ \hat{L}_v(\phi) = \overline{\hat{L}}_v(\phi)(z_*)\ .$$ We will show that $\overline{L}_v(\phi)(z_*) = \overline{\hat{L}}_v(\phi)(z_*)$ for such $z_*$. 

Recall from (\ref{eq:gsharpgprime}) that when $z$ is not a pole of $\phi$, $$\phi^\#(z) = \frac{\max(1, |z|_v)^2}{\max(1, |\phi(z)|_v)^2}\cdot |\phi'(z)|_v\ .$$ This extends continuously to $\pberk\setminus \{\infty, \text{poles of }\phi\}$ as $$[\phi^\#]_z = \frac{\max(1, [T]_z)^2}{\max(1, [\phi]_z)^2}\cdot [\phi']_z\ .$$ Since we are assuming that $\mathcal{J}(\phi)\subseteq \pberk \setminus B_{\vv_\infty}(\zeta_{0,R})$, we can find constants $B_1, B_2$ so that $$B_1 \leq \max(1, [T]_z)^2 \leq B_2$$ for every $z\in \mathcal{J}(\phi)$. By the forward invariance of $\mathcal{J}(\phi)$, this gives that \begin{align}\label{eq:lyapindepbound}\frac{1}{B_2} \leq \frac{1}{\max(1, [T]_{\phi^n(z)})^2} \leq \frac{1}{B_1} \end{align} for every $n$ and for every $z\in \mathcal{J}(\phi)$ (recall that $\mathcal{J}(\phi) \subseteq \pberk \setminus B_{\vv_\infty}(\zeta_{0,R})$ implies that $\phi$ does not have poles in $\mathcal{J}(\phi)$).

By the chain rule, Birkhoff's theorem can be rewritten as $$\overline{L}_v(\phi)(z) = \lim_{n\to\infty} \frac{1}{n} \log_v[(\phi^n)^\#]_z\ ,$$ and similarly for $\overline{\hat{L}}(\phi)(z)$. Since $$\log_v\left(\frac{B_1}{B_2}\right)+\log_v[(\phi^n)']_z \leq \log_v[(\phi^n)^\#]_z \leq  \log_v\left(\frac{B_2}{B_1}\right) + \log_v[(\phi^n)']_z\ ,$$ it follows that $\overline{L}_v(\phi)(z_*) = \overline{\hat{L}}_v(\phi)(z_*)$ as desired.
\end{proof}

\section{Lower Bound on $L_v(\phi)$}\label{sect:mainproof}
At the end of this section we prove Theorem~\ref{thm:lowerboundonLyap}. The proof will make use of a notion of distortion given in \cite{BIJL}, which we recall now.
\subsection{Distortion}
 Given a map $\phi \in K(z)$ and a point $\zeta\in \hberk$, let the distortion of $\phi$ at $\zeta$ be the quantity (see \cite{BIJL}) $$\hat{\delta}(\phi, \zeta): = \log_v \diam_\infty(\zeta) +\log_v [\phi']_\zeta - \log_v [\phi]_\zeta\ .$$ Rearranging this gives 
\begin{equation}\label{eq:logderdistortion}\log_v [\phi']_\zeta = \hat{\delta}(\phi, \zeta) +\log_v[\phi]_\zeta - \log_v \diam_\infty(\zeta)\ .\end{equation}
The following lemma shows that the distortion transforms nicely with respect to pre-composition with elements of $\PGL_2(K)$:
\begin{lemma}\label{lem:preconjugation}
For any $\gamma \in \PGL_2(K)$ and any $\zeta\in \hberk$, we have \begin{equation}\label{eq:deltatransform}\hat{\delta}(\phi\circ \gamma, \zeta) = \hat{\delta}(\phi, \gamma(\zeta))\ .\end{equation}
\end{lemma}
\begin{proof}
The group $\PGL_2(K)$ can be generated by transformations of the form $\gamma(z) = z+a$, $\sigma(z) = az$, and $\tau(z) = \frac{1}{z}$. We will show that (\ref{eq:deltatransform}) holds for each of these types of transformations. 

Suppose first that $\zeta\in\hberk$ is either of type II or III. Given an element  $\gamma\in \PGL_2(K)$, we have
\begin{align}
\hat{\delta}(\phi\circ\gamma, \zeta) &= \log_v \diam_\infty(\zeta) + \log_v [(\phi\circ \gamma)']_\zeta - \log_v [\phi\circ\gamma]_\zeta\nonumber\\
& = \log_v \diam_\infty(\zeta) + \log_v [\phi']_{\gamma(\zeta)} + \log_v [\gamma']_\zeta - \log_v [\phi]_{\gamma(\zeta)}\label{eq:deltabreakdown}\ .
\end{align}

We consider three cases:
\begin{itemize}

\item If $\gamma(z) = z+a$ for some $a\in K$, then $\gamma'(z) = 1$ and the term $\log_v[\gamma']_\zeta$ in (\ref{eq:deltabreakdown}) is zero. Since $\gamma$ is simply a translation, we also have $\diam_\infty(\zeta) = \diam_\infty(\gamma(\zeta))$; inserting this into (\ref{eq:deltabreakdown}) gives $$\hat{\delta}(\phi\circ \gamma, \zeta) = \log_v \diam_\infty (\gamma(\zeta)) + \log_v [\phi']_{\gamma(\zeta)} - \log_v [\phi]_{\gamma(\zeta)} = \hat{\delta}(\phi, \gamma(\zeta))\ .$$

\item If $\sigma(z) = az$ for some non-zero $a\in K$, then $\log_v [\sigma']_{\zeta} = \log_v |a|_v$ and 
\[
\log_v \diam_{\infty}(\sigma(\zeta)) = \log_v \left( |a|_v\cdot \diam_{\infty} (\zeta)\right) = \log_v |a|_v + \log_v \diam_\infty(\zeta)\ .
\] Inserting these into (\ref{eq:deltabreakdown}) gives
\begin{align*}
\hat{\delta}(\phi\circ\sigma, \zeta) &=  \log_v \diam_\infty(\zeta) + \log_v [\phi']_{\sigma(\zeta)} + \log_v |a|_v - \log_v [\phi]_{\sigma(\zeta)}\\
& = \log_v \diam_{\infty}(\sigma(\zeta)) + \log_v [\phi']_{\sigma(\zeta)} -\log_v [\phi]_{\sigma(\zeta)} = \hat{\delta}(\phi, \sigma(\zeta))\ .
\end{align*}

\item Finally, if $\tau(z) = \frac{1}{z}$, note that $\tau'(z) = -(\tau(z))^2$; hence (\ref{eq:deltabreakdown}) becomes 
\begin{equation}\label{eq:deltabreakdownspecial}
\hat{\delta}(\phi\circ \tau, \zeta) = \log_v \diam_\infty(\zeta) + \log_v [\phi']_{\tau(\zeta)} + 2\log_v [\tau]_\zeta - \log_v [\phi]_{\tau(\zeta)}
\end{equation} Write $\zeta = \zeta_{a,r}$. We have two subcases:
\begin{itemize}
\item Suppose that $|a|_v > r$. Then $\tau(D(a,r)) = D\left(\frac{1}{a}, \frac{r}{|a|_v^2}\right)$. By \cite{BR} Corollary 4.2 and our assumption on $|a|_v$, we have
\begin{align*}
[\tau]_{\zeta} = \delta(\tau(\zeta), 0)_\infty &= \max\left(\frac{1}{|a|_v}, \frac{r}{|a|_v^2}\right)\\
& = \frac{1}{|a|_v} \max \left(1, \frac{r}{|a|_v}\right)\\
& = \frac{1}{|a|_v}\ .
\end{align*} Further, $\diam_\infty (\tau(\zeta)) = \frac{r}{|a|_v^2} = \frac{1}{|a|_v^2} \diam_\infty(\zeta)$. Inserting these into (\ref{eq:deltabreakdownspecial}) gives
$$ \hat{\delta}(\phi\circ \tau, \zeta) = \log_v \diam_{\infty}(\zeta) + \log_v [\phi']_{\tau(\zeta)} - 2\log_v |a|_v - \log_v [\phi]_{\tau(\zeta)} = \hat{\delta}(\phi, \tau(\zeta))\ .$$

\item Suppose that $|a|_v \leq r$. Then $\zeta = \zeta_{0, r}$, and hence $\tau(\zeta) = \zeta_{0, \frac{1}{r}}$ (see, e.g. \cite{BR} Lemma 2.4). This implies $[\tau]_\zeta = \delta(\tau(\zeta), 0)_\infty = \frac{1}{r}$. We also have that $\diam_{\infty}(\tau(\zeta)) = \frac{1}{r} = \frac{1}{\diam_\infty(\zeta)}$. Inserting these into (\ref{eq:deltabreakdownspecial}) gives 
$$\hat{\delta}(\phi \circ \tau, \zeta) = \log_v \diam_{\infty}(\zeta) + \log_v [\phi']_{\tau(\zeta)} - 2\log_v r - \log_v [\phi]_{\tau(\zeta)} = \hat{\delta}(\phi, \tau(\gamma))\ .$$
\end{itemize}

Note that the functions $\log_v \diam_\infty(\zeta), \log_v [\phi]_\zeta,$ and $\log_v [\phi']_\zeta$ are continuous with respect to the strong metric on $\hberk$; consequently the asserted equality holds also for type IV points as well, since type II and type III points are dense in $\hberk$.

\end{itemize}

\end{proof}

We can apply the previous lemma to show that, in the case that $\textrm{char}(K) = 0$, the distortion is essentially bounded below by $\kappa = \log_v \min( |i|_v\ : \ 1\leq i \leq d)$. This was noted already in \cite{BIJL} Remark 3.4 under certain hypotheses on $\zeta$ and $\phi$; we identify an error term that will account for the cases where the hypotheses are not met, and use this to establish the following lower bound on the log of the derivative:
\begin{lemma}\label{lem:hatdeltabound}
Assume $\textrm{char}(K)=0$. Fix $\zeta\in \hberk$, and let $\kappa = \log_v \min(|i|_v\ : \ 1\leq i \leq d)$. Then for any $\zeta\in \hberk$, we find 
\begin{equation}\label{eq:logphiepsilonbound}\log_v [\phi']_\zeta \geq \kappa  +\log_v \diam_\infty(\phi(\zeta)) - \log_v \diam_\infty(\zeta)\ .\end{equation} 
\end{lemma}
\begin{proof}

Assume $\zeta$ is of type II or III, and write $\zeta = \zeta_{a,r}$, where the center $a$ is chosen so that the open disc $D(a,r)^-$ does not contain any poles of $\phi$. Let $b=\phi(a)$, and write $\phi(\zeta) = \zeta_{b,s}$ for some $s>0$. We consider the map $\phi(z) - b$. 

Note that $N_{\vv_a}^-(\phi-b, \zeta, \infty) = 0$, as we chose $a$ so that $D(a,r)^-$ contains no poles of $\phi$. We also find $N_{\vv_a}^-(\phi-b, \zeta, 0) = N_{\vv_a}^-(\phi, \zeta, b)\in \{1,2,...,d\}$. Then by \cite{BIJL} Lemma 3.3, we find
\begin{align*}
\hat{\delta}(\phi-b, \zeta) &\geq \log_v |N_{\vv_a}^-(\phi-b, \zeta, \infty) - N_{\vv_a}^-(\phi-b, \zeta, 0)|_v\\
& = \log_v |N_{\vv_a}^-(\phi, \zeta, b)|_v \geq \kappa\ .
\end{align*} Note also that 
\[
\log_v[\phi-b]_\zeta = \log_v \diam_\infty(\phi(\zeta))\ .
\] Consequently, we have
\begin{align*}
\log_v [\phi']_\zeta = \log_v [(\phi-b)']_\zeta &= \hat{\delta}(\phi-b, \zeta) + \log [ \phi-b]_\zeta - \log_v \diam_\infty(\zeta)\\
& \geq \kappa + \log_v \diam_\infty(\phi(\zeta)) - \log_v \diam_\infty(\zeta)\ 
\end{align*}as asserted.

Now observe that the functions $\log_v [\phi'], \log_v \diam_\infty(\cdot), $ and $\log_v \diam_\infty(\phi(\cdot))$ are all continuous with respect to the strong topology; consequently, the asserted inequality extends to type IV points as well, since the type II and type III points are dense in $\hberk$ with respect to the strong topology.

\end{proof}

The idea now is to integrate both sides of (\ref{eq:logphiepsilonbound}) against the measure $\mu_\phi$ and use the $\phi$-invariance of $\mu_\phi$ to `cancel' the diameter terms appearing on the right side of the inequality. However, this will not work for general measures $\mu_\phi$: on the one hand, we do not know whether $\int_{\pberk} \log_v\diam_\infty(\zeta) d\mu_\phi$ is finite in general; on the other hand, the bounds in Lemma~\ref{lem:hatdeltabound} apply only when $\zeta\in \hberk$, and it is possible that $\mu_\phi$ is supported only at type I points! 

\subsection{Proof of Theorem~\ref{thm:lowerboundonLyapbdd}}\label{sect:pfthmtwo}
If $\supp(\mu_\phi)\subseteq \hberk$ and $\int_{\pberk} \log_v \diam_\infty(\zeta)d\mu_\phi(\zeta) < \infty$, we \textit{can} integrate both sides of (\ref{eq:logphiepsilonbound}) against $\mu_\phi$ to establish Theorem~\ref{thm:lowerboundonLyapbdd}: \begin{proof}[Proof of Theorem~\ref{thm:lowerboundonLyapbdd}]
Since we are assuming $\supp(\mu_\phi)= \mathcal{J}(\phi)\subseteq \hberk$, (\ref{eq:logphiepsilonbound}) holds for all $\zeta\in \supp(\mu_\phi)$. As we are also assuming that $\log_v\diam_\infty(\cdot)\in L^1(\mu_\phi)$, we may integrate both sides of (\ref{eq:logphiepsilonbound}) and using the $\phi$-invariance of $\mu_\phi$ gives
\begin{align*}
\int \log_v [\phi']_\zeta d\mu_\phi(\zeta) & \geq \kappa + \int\log_v \diam_\infty(\phi(\zeta)) d\mu_\phi - \int \log_v \diam_\infty(\zeta) d\mu_\phi\\
& = \kappa + \int \log_v \diam_\infty(\zeta) d\mu_\phi - \int \log_v \diam_\infty(\zeta) d\mu_\phi\\
& = \kappa\ .
\end{align*} This gives the asserted lower bound, since $L_v(\phi) = \hat{L}_v(\phi) = \int \log_v [\phi']_\zeta d\mu_\phi(\zeta) $.

\end{proof}

\noindent The following example shows that the lower bound obtained in Theorem~\ref{thm:lowerboundonLyapbdd} is sharp in some degrees $d$:\\

\noindent\emph{Example:} Fix a prime number $p$. Let $K=\mathbb{C}_p$ be the field of $p$-adic numbers, and let $\phi(z) = z^d$, where $d=ap^\ell$ for $a\in \{1, ...,p-1\}$ coprime to $p$, and $\ell>0$. Here, $$\kappa = \min\left(\log_v |m|_v\ : \ 1\leq m \leq ap^\ell\right) = \log_v |p^\ell|_v =-\ell <0\ .$$ The map $\phi$ has good reduction, hence the equilibrium measure for $\phi$ is $\mu_\phi = [\zetaG]$. Then $$\hat{L}(\phi) = \int \log_v [\phi']_z d\mu_\phi = \log_v [ap^\ell z^{d-1}]_{\zetaG} = \log_v (|ap^\ell|_v\cdot [z^{d-1}]_{\zetaG})\ .$$ Since $\psi(z) = z^{p-1}$ has non-constant reduction, we have $[z^{p-1}]_{\zetaG} = [z]_{\psi(\zetaG)} = [z]_{\zetaG} = 1$. Hence $$\hat{L}(\phi) = \log_v (|ap^\ell|_v\cdot [z^{p-1}]_{\zetaG}) = \log_v |p^\ell|_v=\kappa \ .$$ By Proposition~\ref{prop:equalLyapunovs} above, we saw that $L(\phi) = \hat{L}(\phi)$, which establishes the asserted sharpness.\\

However, it is worth noting that often this lower bound is not sharp, even for some maps of potential good reduction. When the degree $d> p$ is coprime to $p$, maps $\phi$ of potential good reduction satisfy $\hat{L}(\phi) = 0$ (see Proposition~\ref{prop:goodredlyapzero} below); but if $d$ is sufficiently large, then $\kappa \neq 0$, so that $\hat{L}(\phi) > \kappa$. 

\subsection{Equidistribution}\label{sect:eq}
The proof of Theorem~\ref{thm:lowerboundonLyap} will rely on two equidistribution theorems, which we formulate and prove here.

\begin{prop}\label{prop:wildintegrallowerbound}
Let $\nu_n = \frac{1}{d^n} (\phi^{n})^*[\zetaG]$ be the probability measure supported at the preimages of $\zetaG$ weighted with multiplicity. Let $\mathcal{L}_\phi$ be a Lipschitz constant for the action of the rational map $\phi$ on $\mathbb{P}^1(K)$ with respect to the spherical metric. Then for all $n\geq 1$, $$\int_{\pberk} \log_v \diam_{\infty}(\zeta) d(\nu_{n-1} - \nu_n)(\zeta) \geq -(d+1)\log_v \mathcal{L}_\phi\ .$$
\end{prop}

Before proving Proposition~\ref{prop:wildintegrallowerbound}, we need the following lemma:

\begin{lemma}\label{lem:lipg}
Fix a type II point $\zeta_0\in \hberk$. Choose $\gamma\in \PGL_2(K)$ with $\gamma(\zetaG) = \zeta_0$, and let $\mathcal{L}_{\phi^\gamma}$ denote the Lipschitz constant for the action of $\phi^\gamma$ on $\mathbb{P}^1(K)$ in the spherical metric. Let 
\begin{equation}\label{eq:g1} g_1(\cdot, \zeta_0) = \frac{1}{d} \sum_{\phi(\zeta_i) = \zeta_0} \langle \cdot, \zeta_i\rangle_{\zeta_0}\ ,\end{equation} which is a potential for the measure $\frac{1}{d}\phi^*[\zeta_0] - [\zeta_0]$. Then $$\sup_{z\in \pberk} |g_1(z, \zeta_0)| \leq \log_v \mathcal{L}_{\phi^\gamma}\ .$$
\end{lemma}

\begin{proof}
Note that as a function of $z$, the potential $g_1(z, \zeta_0)$ is constant off of the tree spanned by $\zeta_0$ and the points $\zeta_i$. So it will suffice to determine $\sup_{z\in \pberk} |\langle z, \zeta_i\rangle_{\zeta_0}|$ for each $\zeta_i\in \phi^{-1}(\zeta_0)$; this supremum occurs when $z=\zeta_i$, and the upper bound in \cite{KJ2} Lemma~1 gives that $|g_1(z, \zeta_0)| \leq \log_v \mathcal{L}_{\phi^\gamma}$ as asserted.

\end{proof}

We are now ready to prove Proposition~\ref{prop:wildintegrallowerbound}:

\begin{proof}[Proof of Proposition~\ref{prop:wildintegrallowerbound}]

Recall (\cite{FRLErgodic}, Proposition-D\'efinition 3.1) that for any fixed $\zeta_0\in \hberk$, the measures $\nu_n$ satisfy $$ \nu_n = [\zeta_0] + \Delta g_n(\cdot, \zeta_0)\ ,$$ where $\Delta$ is the Laplacian on $\pberk$ and $$g_n(\cdot, \zeta_0) = \sum_{k=0}^{n-1} \frac{g_1(\phi^k(\cdot), \zeta_0)}{d^k}\ ,$$ and $g_1(\cdot, \zeta_0)$ is the potential function given in Lemma~\ref{lem:lipg}.

Fix $\zeta_0 = \zetaG$, and for $n\geq 0$ define
\[
I_n = \int_{\pberk} \log_v \diam_\infty(\zeta) d(\nu_{n-1} - \nu_n)(\zeta)\ .
\] Let $\Gamma_n$ be the tree spanned by the points of $\phi^{-n}(\zetaG) \cup \phi^{-(n-1)}(\zetaG)$; note that $\Gamma_n$ contains the support of $\nu_{n-1} - \nu_n$. We first show that our integral can be restricted to the tree $\Gamma_n$.

The function $\log_v \diam_\infty(\zeta)$ is not continuous, but its retraction $\log_v \diam_\infty(r_{\Gamma_n}(\zeta))$ is continuous and of bounded differential variation on $\pberk$ and constant on branches off of $\Gamma_n$ (here $r_{\Gamma_n}$ is the retraction map defined in Section~\ref{sect:tree} above). Since the measure $\nu_{n-1} - \nu_n$ has finite support in $\Gamma$, we find 
\begin{equation*}
I_n = \int\log_v \diam_\infty(r_{\Gamma_n}(\zeta)) d(\nu_{n-1} - \nu_n )(\zeta)\ .
\end{equation*}The difference of measures appearing above can be written as $\nu_{n-1} - \nu_n = \Delta(g_{n-1}(\cdot, \zetaG) - g_n(\cdot, \zetaG))$. Applying the definition of the Berkovich Laplacian as a limit of coherent measures (see \cite{BR} Proposition 5.10), we find
\begin{equation*}
\int \log_v \diam_\infty(r_{\Gamma_n}(\zeta)) d\Delta(g_{n-1} - g_n)(\zeta) = \lim_{\substack{\longrightarrow\\ \Gamma}} \int_\Gamma \log_v \diam_\infty(r_{\Gamma_n}(\zeta))\Delta_{\Gamma}(g_{n-1}-g_n)(\zeta)\ .
\end{equation*}

Since the measure $\Delta(g_{n-1}-g_n) = \nu_{n-1} - \nu_n$ has finite support which is contained in $\Gamma_n$, the limit stabilizes at $\Gamma=\Gamma_n$, and we recover 

\begin{equation}\label{eq:restrictedtotree}
I_n = \int_{\Gamma_n} \log_v \diam_\infty(\zeta) d\Delta_{\Gamma_n} (g_{n-1} - g_n)(\zeta)\ .
\end{equation}

Using the expression for $g_n(\cdot, \zetaG)$ as a finite sum, $\Delta_{\Gamma_n}(g_{n-1}-g_n)$ can be written 
\begin{equation*}
\Delta_{\Gamma_n}(g_{n-1} - g_n) = -\Delta_{\Gamma_n} \frac{g_1(\phi^{n-1}(\cdot, \zetaG))}{d^{n-1}}\ .
\end{equation*} Inserting this into (\ref{eq:restrictedtotree}) and interchanging the Laplacian, we find 
\begin{equation}\label{eq:intlogdiamswappedLap}
I_n = -\int_{\Gamma_{n}} \frac{g_1(\phi^{n-1}(\zeta), \zetaG)}{d^{n-1}} d\Delta_{\Gamma_{n}} \log_v \diam_{\infty}(\cdot)\ .
\end{equation} 
Recall that the diameter can be expressed in terms of the Hsia kernel by $\diam_\infty(\zeta) = \delta(\zeta, \zeta)_{\infty}$. In \cite{KJ1} Lemma~6 it was shown that 
\begin{equation}\label{eq:Deltalogdiam}
\Delta_\Gamma \log_v \delta(\cdot, \cdot)_{\infty} =  -2\mu_{\textrm{Br}, \Gamma} + 2[r_\Gamma(\infty)]\ ,
\end{equation} where $\mu_{\textrm{Br}, \Gamma} = \frac{1}{2}\sum_{P\in \Gamma} (2-v_\Gamma(P))[P]$ is the branching measure of the finite graph $\Gamma$ and $r_\Gamma(\infty)$ is the retraction of $\infty$ to $\Gamma$. The Jordan decomposition for the branching measure can be given explicitly as $$\mu_{\textrm{Br}, \Gamma} = \mu_{\textrm{Br}, \Gamma}^+ + \mu_{\textrm{Br}, \Gamma}^-\ ,$$ where $$\mu_{\textrm{Br}, \Gamma}^+ = \frac{1}{2}\sum_{v(P) =1 } [P]$$ is a positive measure supported on the endpoints of $\Gamma$, and $$\mu_{\textrm{Br}, \Gamma}^- = \frac{1}{2} \sum_{v(P) >1} (2-v(P))[P]$$ is either the zero measure, or is a negative measure supported on the interior branch points of $\Gamma$. The measure $\mu_{\textrm{Br}, \Gamma}$ is a probability measure, and we record for later use that
\begin{equation}\label{eq:branchingmeasuredecomp}
-\mu_{\textrm{Br}, \Gamma}^-(\Gamma) = \mu_{\textrm{Br}, \Gamma}^+(\Gamma) - 1\ .
\end{equation}

Inserting (\ref{eq:Deltalogdiam}) into (\ref{eq:intlogdiamswappedLap}) and using the Jordan decomposition for the measure $\mu_{\textrm{Br}, \Gamma_n}$ gives 
\begin{align}
I_n &= \int_{\Gamma_n}\frac{-g_1(\phi^{n-1}(\zeta), \zetaG)}{d^{n-1}} d(-2\mu_{\textrm{Br}, \Gamma_n}^+ - 2 \mu_{\textrm{Br}, \Gamma_n}^-)(\zeta)-2\frac{g_1(\phi^{n-1}(r_{\Gamma_n}(\infty)), \zetaG)}{d^{n-1}} \label{eq:integralwithJordandecomp} \ .
\end{align}  By Lemma~\ref{lem:lipg} and the fact that $g_1(\cdot, \zetaG) \geq 0$, we find that $0\leq g_1(\cdot, \zetaG) \leq \log_v \mathcal{L}_\phi$; hence the integrand in (\ref{eq:integralwithJordandecomp}) is negative. Since $\mu_{\textrm{Br}, \Gamma_n}^+$ is a positive measure, we obtain a lower bound in (\ref{eq:integralwithJordandecomp}) by removing its contribution from the expression, i.e.
\begin{equation*}
I_n \geq \int_{\Gamma_n} \frac{-g_1(\phi^{n-1}(\zeta), \zetaG)}{d^{n-1}}d(-2\mu_{\textrm{Br}, \Gamma_n}^-) - 2\frac{g_1(\phi^{n-1}(r_{\Gamma_n}(\infty)), \zetaG)}{d^{n-1}}\ .
\end{equation*} 
Applying the bound $-g_1(\cdot, \zetaG) \geq -\log_v \mathcal{L}_\phi$, this becomes
\begin{align*}
I_n&\geq \frac{-\log_v \mathcal{L}_\phi}{d^{n-1}} \left( \int_{\Gamma_n} d(-2\mu_{\textrm{Br}, \Gamma_n}^-)(\zeta) + 2\right)\\
& = -\frac{\log_v \mathcal{L}_\phi}{d^{n-1}} \left(-2\mu_{\textrm{Br}, \Gamma_n}^- (\Gamma_n) +2\right)\ .
\end{align*}
Inserting the expression (\ref{eq:branchingmeasuredecomp}), we find 
\begin{equation}\label{eq:lowerboundmuplus}
I_n\geq -\frac{\log_v \mathcal{L}_\phi}{d^{n-1}} \cdot 2\mu_{\textrm{Br}, \Gamma_n}^+(\Gamma_n)\ .
\end{equation}

We observe that $2\mu_{\textrm{Br}, \Gamma}^+(\Gamma)$ counts the number of endpoints in $\Gamma$; since $\Gamma_n$ is spanned by the points $\phi^{-n}(\zetaG) \cup \phi^{-(n-1)}(\zetaG)$, there are at most $d^n+d^{n-1}$ endpoints. Inserting this bound into (\ref{eq:lowerboundmuplus}) gives $$I_n \geq -\frac{\log_v \mathcal{L}_\phi}{d^{n-1}} (d^n+d^{n-1}) = -(d+1)\log_v \mathcal{L}_\phi$$ which is the desired result.

\end{proof}

The next result is a logarithmic equidistribution theorem that extends the equidistribution of preimages given in \cite{FRLErgodic}; the proof idea is a slight modification of the proof given in \cite{FRLErgodic} Proposition-D\'efinition 3.1.

\begin{prop}\label{prop:quantlogeq}
Let $K$ be a complete, algebraically closed non-Archimedean valued field, and let $\phi\in K(z)$. Fix a type II point $\xi\in \hberk$, and let $\gamma\in \PGL_2(K)$ be such that $\gamma(\zetaG) = \xi$. Denote by $\mathcal{L}_{\phi}$ the Lipschitz constant for the action of $\phi$ on $\PP^1(K)$ in the spherical metric.

Let $\nu_n = \frac{1}{d^n} (\phi^{n})^*[\xi]$, and fix a point $a\in \pberk$. We have $$\left|\int_{\pberk} \log_v \delta(\cdot, a)_{\zetaG} d(\nu_n -\mu_\phi)\right| \leq \frac{2d \log_v \mathcal{L}_{\phi^\gamma}}{d-1} \cdot \frac{1}{d^n}\ .$$
\end{prop}
\begin{proof}

Fix $\epsilon \in (0, 1)$; if $\diamG(a) \geq \epsilon$ let $a^\epsilon = a$; otherwise, let $a^\epsilon$ be the unique point on the segment $[a^\epsilon, \zetaG]$ with $\diamG(a^\epsilon) = \epsilon$. 

We will draw from the proof of Proposition-D\'efinition 3.1 of \cite{FRLErgodic}. Fix a type II point $\xi\in \hberk$, and let $g_1$ be the potential for the measure $\nu_1=\frac{1}{d} \phi^*[\xi] - [\xi]$ given in Lemma~\ref{lem:lipg}, so that $\Delta g_1 = \frac{1}{d} \phi^*[\xi]-[\xi] $. Let $$g_n(\cdot, \xi) = \sum_{k=0}^{n-1} \frac{g_1(\phi^k(\cdot), \xi)}{d^k}\ ,$$ whereby $\Delta g_n(\cdot,\xi) = \frac{1}{d^n} (\phi^{n})^* [\xi]- [\xi]= \nu_n-[\xi]$. Since $g_1$ is bounded on $\pberk$,  the functions $g_n(\cdot, \xi)$ converge uniformly to a function $$g_\infty(\cdot, \xi)= \sum_{k=0}^\infty \frac{g_1(\phi^k(\cdot), \xi)}{d^k}\ ,$$ and the measures $\Delta g_n(\cdot, \xi)$ converge weakly to the measure $\mu_\phi-[\xi]$. In particular, $g_n(\cdot, \xi) - g_\infty(\cdot, \xi)$ is a potential for the measure $\nu_n - \mu_\phi$. We remark that $g_n(\cdot, \xi) - g_\infty(\cdot, \xi)$ is continuous on all of $\pberk$.

The integrand $\log_v \delta(\cdot, a^\epsilon)_{\zetaG}$ is continuous on $\pberk$, and therefore we find (see \cite{BR} Corollary 5.39):
\begin{align}
\left| \int \log_v \delta(z, a^\epsilon)_{\zetaG} d(\nu_n - \mu_\phi)\right| & = \left| \int \log_v \delta(z, a^\epsilon)_{\zetaG} d \Delta (g_n(\cdot, \xi) - g_\infty(\cdot, \xi))\right|\nonumber\\
& = \left| \int  g_n(\cdot, \xi) - g_\infty(\cdot, \xi)\  d\Delta \log_v \delta(z, a^\epsilon)_{\zetaG}\right|\nonumber\\
& = \left| \int g_n (\cdot, \xi) - g_\infty(\cdot, \xi)\ d \left([a^\epsilon] - [\zetaG]\right)\right|\nonumber\\
& \leq \left| g_n(a^\epsilon, \xi) - g_\infty(a^\epsilon, \xi)\right| + \left| g_n(\zetaG , \xi) - g_\infty(\zetaG, \xi)\right|\nonumber\\
& \leq 2 \sum_{k=n}^\infty \frac{\sup|g_1(\cdot, \xi)|}{d^k} = \frac{2d\sup |g_1(\cdot, \xi)| }{d-1}\cdot \frac{1}{d^n} \nonumber\ .
\end{align} Let $\gamma\in \PGL_2(K)$ be such that $\gamma(\zetaG) = \xi$. Estimating $\sup |g_1(\cdot, \xi)|$ as in Lemma~\ref{lem:lipg}, we obtain

\begin{equation}\label{eq:boundforsecondintegral}
\left| \int \log_v \delta(z, a^\epsilon)_{\zetaG} d(\nu_n - \mu_\phi)\right| \leq \frac{2d\log_v \mathcal{L}_{\phi^\gamma}}{d-1}\cdot\frac{1}{d^n}\ .
\end{equation} Note that the left side of (\ref{eq:boundforsecondintegral}) is $|u_{\nu_n - \mu_\phi}(a^\epsilon, \zetaG)|$, where $u_{\nu_n - \mu_\phi}(\cdot, \zeta)$ is the potential function for $\nu_n - \mu_\phi$ (see Section~\ref{sect:potentialfunctions}). Since $\nu_n - \mu_\phi$ has continuous potentials, the result follows by letting $\epsilon \to 0$.
\end{proof}

As a corollary, we have
\begin{cor}\label{cor:convergencebracket}
Let $\phi\in K(z)$ have degree $d\geq 2$. Fix a point $\xi\in \hberk$ and let $\nu_n = \frac{1}{d^n}\phi^{n*}[\xi]$. Let $0\neq \psi \in K(z)$ be any rational function. Then $$\int \log_v [\psi]_\zeta d\nu_n(\zeta) \to \int \log_v [\psi]_\zeta d\mu_\phi(\zeta)$$ as $n\to\infty$.
\end{cor}
\begin{proof}
Write $\textrm{div}(\psi) = \sum_{i=1}^m n_i (a_i)$ for the divisor of $\psi$, a formal sum of the roots and poles $\{a_i\}$ of $\psi$. By \cite{BR} Corollary 4.14, there exists a constant $C=C(\phi, \zetaG)$ such that \begin{equation}\label{eq:psibracketdecomp}[\psi]_\zeta = C\prod_{i=1}^m \delta(\zeta, a_i)^{n_i}_{\zetaG}\ .\end{equation} Therefore, $$\int \log_v [\psi]_\zeta d\nu_n = \log_v C + \sum_{i=1}^m n_i\int \cdot \log_v \delta(\zeta, a_i)_{\zetaG} d\nu_n(\zeta)\ .$$ The result now follows by applying the convergence in Proposition~\ref{thm:logeq} to each of the integrals appearing in the sum, and reassembling the terms using (\ref{eq:psibracketdecomp}).

\end{proof}

We remark that Okuyama has established a similar result when the integrand is $\log_v [\phi^\#]$, and has used these results to obtain approximations to the Lyapunov exponent; see \cite{YO} and \cite{YO2}.

\subsection{Lower Bound for $L(\phi)$: General Case}\label{sect:generalcase}

We are now in a position to prove the Theorem~\ref{thm:lowerboundonLyap}:

\begin{proof}[Proof of Theorem~\ref{thm:lowerboundonLyap}]
Let $\nu_n= \frac{1}{d^n}(\phi^{n})^* [\zetaG]$ be the measures supported at the $n$-th order preimages of $\zetaG$ weighted according to multiplicity. Integrating (\ref{eq:logphiepsilonbound}) against $\nu_n$ we have
\begin{equation}\label{eq:preequidistribution}
\int \log_v [\phi']_\zeta d\nu_n \geq \kappa + \int \log_v \diam_{\infty}(\phi(\zeta)) d\nu_n- \int \log_v \diam_\infty(\zeta) d\nu_n\ .
\end{equation} Since the measures $\nu_n$ satisfy $\phi_* \nu_n = \nu_{n-1}$, the middle integral can be written as $$\int \log_v \diam_\infty(\phi(\zeta)) d\nu_n = \int \log_v \diam_\infty(\zeta) d\nu_{n-1}\ .$$ 
Hence, (\ref{eq:preequidistribution}) becomes 
\begin{equation}\label{eq:preequidistributionI}
\int \log_v [\phi']_\zeta d\nu_n \geq \kappa + \int \log_v \diam_\infty(\zeta) d(\nu_{n-1}- \nu_n)(\zeta)\ .
\end{equation} 
Proposition~\ref{prop:wildintegrallowerbound} gives
\begin{equation}\label{eq:setlowerbound}\int \log_v [\phi']_\zeta d\nu_n \geq \kappa - (d+1)\log_v \mathcal{L}_\phi\ , \ \forall n\geq 1\ ;\end{equation}passing to a limit, the left side becomes $\int \log_v [\phi']_\zeta d\mu_\phi$, which equals $L_v(\phi)$ by Proposition~\ref{prop:equalLyapunovs}. We have therefore shown
\[
L(\phi) \geq \kappa - (d+1) \log_v \mathcal{L}_\phi\ .
\] Since $L(\phi)$ is invariant under conjugation (Lemma~\ref{lem:coordchangeok}), Theorem~\ref{thm:lowerboundonLyap} follows by taking infima over all conjugates $\phi^\gamma$ for $\gamma\in\PGL_2(K)$.
\end{proof}

\subsection{Good and Separable Reduction}
Finally, we show that in the case that $\phi$ has potential good reduction at $\zeta$ and also separable reduction at $\zeta$, the Lyapunov exponent for $\phi$ is necessarily 0. The proof is entirely independent of the machinery developed above.

\begin{prop}\label{prop:goodredlyapzero}
If $\phi$ has good reduction at $\zeta$, and the reduction at $\zeta$ is also separable, then $$L_v(\phi)=0\ .$$
\end{prop}
\begin{proof}
We may first conjugate $\phi$ so that it has good and separable reduction at $\zetaG$. By Lemma~\ref{lem:coordchangeok}, this is leaves $L_v(\phi)$ unaffected. Having done this, $\mu_{\phi} = [\zetaG]$, and $L_v(\phi) = \log_v [\phi']_{\zetaG}$. Since $\phi$ has separable reduction, it follows from a lemma of Xander Faber (\cite{XF}, Lemma 6.3) that $\widetilde{\phi'} \neq 0$. We further claim that $\widetilde{\phi'}$ is not constant. To see this, write $\phi = \frac{f(z)}{g(z)}$ with $(f,g) = 1$. Since $\phi$ has good reduction, we also have that $(\tilde{f}, \tilde{g}) = 1$. Suppose $\tilde{\phi'} = C$ for some nonzero $C\in k$. Then we find $$\tilde{f'}\tilde{g} - \tilde{g'} \tilde{f} = C\tilde{g}^2\ .$$ Rewriting this gives $$\tilde{g} (\tilde{f'} - C\tilde{g}) = \tilde{g'} \tilde{f}\ .$$ Note that $\tilde{g'}$ necessarily has smaller degree than $\tilde{g}$, and as such we find that some root of $\tilde{g}$ is also a root of $\tilde{f}$, contradicting that $(\tilde{f}, \tilde{g}) = 1$ has good reduction. If $\tilde{\phi'} \equiv \infty$, then necessarily $\tilde{g} \equiv 0$, which is impossible since $\phi$ has good reduction.

Since $\tilde{\phi'}$ is not constant, the map $\phi'$ fixes $\zetaG$, hence $$L_v(\phi) = \log_v [\phi']_{\zetaG} = \log_v [T]_{\phi'(\zetaG)} = \log_v [T]_{\zetaG} = 0\ .$$

\end{proof}

\bibliographystyle{plain}

\end{document}